\documentclass[11pt,reqno]{amsart}
\usepackage[a4paper, hmargin={2.7cm,2.7cm},vmargin={3.3cm,3.3cm}]{geometry}
\usepackage[english]{babel}
\usepackage{color}
\usepackage[noadjust]{cite}
\usepackage{amsmath}
\usepackage{amssymb}
\usepackage{amsfonts}
\usepackage{amsthm}
\usepackage{dsfont}
\usepackage{enumerate}
\usepackage{amsthm}
\usepackage{todonotes}
\usepackage{hyperref}
\usepackage{etoolbox, url}

\hypersetup{
	colorlinks   = true, 
	urlcolor     = blue, 
	linkcolor    = blue, 
	citecolor   = red 
}

\numberwithin{equation}{section}

\makeatletter
\@namedef{subjclassname@2020}{\textup{2020} Mathematics Subject Classification}
\makeatother

\makeatletter
\patchcmd{\ttlh@hang}{\parindent\z@}{\parindent\z@\leavevmode}{}{}
\patchcmd{\ttlh@hang}{\noindent}{}{}{}
\makeatother

\theoremstyle{plain}
\newtheorem{theorem}{Theorem}

\theoremstyle{definition}

\theoremstyle{remark}

\DeclareMathOperator{\vol}{vol}
\DeclareMathOperator{\Span}{span}

\title{Comment on ``Affine density, von Neumann dimension and a problem of Perelomov''}

\author{Jos\'e Luis Romero}

\address{Faculty of Mathematics, University of Vienna, Oskar-Morgenstern-Platz 1, A-1090 Vienna, Austria.}
\email{jose.luis.romero@univie.ac.at}

\begin{document}

\maketitle

\begin{abstract}
 We point out that the main theorem of Ref2 := [Adv. Math. 407, Article ID 108564, 22 p. (2022)] is included in the prior research survey Ref1 := [Expo. Math., 40(2), 265-301, 2022]. For context, we also reproduce the rather simple proof given in Ref1 (essentially due to Janssen). We also offer a brief discussion of the literature that is alternative to that given in Ref2. Specifically, we differ with Ref2 with respect to novelty, relevance, and attribution of stated results (and, to some extent, correctness of statements).
\end{abstract}

\section{Comparison of presented results}
Our research survey article \cite{romero2022density} contains the following statement concerning a locally compact, second countable unimodular group $G$.
\begin{theorem}[{\cite[Theorem 7.4]{romero2022density}}] \label{th_rvv}
	Let $\Gamma \subseteq G$ be a lattice and let
	$(\pi, \mathcal{H}_\pi)$ be a discrete series $\sigma$-representation of $G$ of formal dimension $d_{\pi} > 0$.
	\begin{enumerate}[(i)]
		\item[(i)] If $\pi|_{\Gamma}$ admits a cyclic vector, then $\vol (G / \Gamma) d_{\pi} \leq 1$.
		
		\noindent In particular, if $\pi|_{\Gamma}$ admits a frame vector, then  $\vol (G / \Gamma) d_{\pi} \leq 1$.
		
		\item[(ii)] If $\pi|_{\Gamma}$ admits a Riesz vector, then  $\vol (G / \Gamma) d_{\pi} \geq 1$.
		
		\item[(iii)] Suppose $(\Gamma, \sigma)$ satisfies Kleppner's condition.
		If $\pi(\Gamma)''$ admits a separating vector, then $\vol (G / \Gamma)d_{\pi} \geq 1$.
	\end{enumerate}
\end{theorem}
Theorem \ref{th_rvv} is the first half of what we called ``the density theorem''; only parts (i) and (ii) are relevant for this letter.

In the subsequent article \cite{abreu2022affine} the authors define the following projective representation $\tau_n^{\alpha}$ of the group $\mathrm{PSL}(2,\mathbb{R})$ acting on the space $\mathcal{W}_{\psi_n^{\alpha}}=\mathcal{W}_{\psi_n^{\alpha}}(\mathbb{C}^+)$, which is the image of the Hardy space $H^2(\mathbb{R})$ under the wavelet transform with a certain special window function $\psi_n^\alpha \in H^2(\mathbb{R})$:
\begin{align} \label{eq:rep_wavelet_space}
	\tau_n^{\alpha} (m^{-1}) F (z)  = \bigg( \frac{|cz + d|}{cz+d} \bigg)^{2n + \alpha + 1} F(m \cdot z), \quad z \in \mathbb{C}^+, \;
	m = \begin{pmatrix}
		a & b \\
		c & d
	\end{pmatrix}
	\in \mathrm{PSL}(2,\mathbb{R}).
\end{align}
The following is presented as the main result in \cite{abreu2022affine}\footnote{In fact, in \cite{abreu2022affine}, the right-hand of \eqref{eq:rep_wavelet_space} is used as the definition of $\tau_n^{\alpha} (m) F(z)$, which would not define a projective representation. We understand this to be an error.}.

\begin{theorem}[{\cite[Theorem 2.1]{abreu2022affine}}] \label{th_as}
Let $\Gamma \subset \mathrm{PSL}(2,\mathbb{R})$ be a Fuchsian group
with fundamental domain $\Omega \subset \mathbb{C}^{+}$ (of finite volume)
and $F\in \mathcal{W}_{\psi_n^\alpha}$. If $\{\tau_{n}^{\alpha
}(\gamma)F\}_{\gamma \in \Gamma }$ is a frame for $\mathcal{W}_{\psi
	_{n}^{\alpha }}$, then $\Omega $ is compact \footnote{{\bf The claim about compactness is not correct; see below.}} and 
\begin{equation}
	\left\vert \Omega \right\vert \leq \frac{C_{\psi _{n}^{\alpha }}}{\Vert \psi
		_{n}^{\alpha }\Vert _{2}^{2}}=\frac{2}{\alpha }\text{,}  \label{eq:thm-sam}
\end{equation}
where $|\Omega |$ is calculated via the measure $\mu$.

If $\{\tau_{n}^{\alpha }(\gamma)F\}_{\gamma \in \Gamma }$ is a Riesz
sequence for $\mathcal{W}_{\psi _{n}^{\alpha }}$, then 
\begin{equation}
	\left\vert \Omega \right\vert \geq \frac{C_{\psi _{n}^{\alpha }}}{\Vert \psi
		_{n}^{\alpha }\Vert _{2}^{2}}=\frac{2}{\alpha }\text{.}  \label{eq:thm-int}
\end{equation}
\end{theorem}
Let us start by mentioning that in \cite[Example 9.2]{romero2022density} we discuss at length how Theorem \ref{th_rvv} applies to the so-called holomorphic discrete series of $\mathrm{PSL}(2,\mathbb{R})$, which exhaust up to complex conjugation and projective unitary equivalence all square integrable irreducible representations of $\mathrm{PSL}(2,\mathbb{R})$ (including \eqref{eq:rep_wavelet_space}).

\subsection{The formal dimension and the admissibility constant are reciprocals}
It is instructive to also compare Theorem \ref{th_rvv} with $G:=\mathrm{PSL}(2,\mathbb{R})$ to Theorem \ref{th_as} directly: we first sort out the jargon and notation. A
$\sigma$-representation is a projective representation (with cocycle $\sigma$). The Hardy space $H^2(\mathbb{R})$ consists of all $f\in L^2(\mathbb{R})$ with Fourier transform $\hat{f}$ supported on $[0,\infty)$.
A Fuchsian group $\Gamma \subset \mathrm{PSL}(2,\mathbb{R})$ is just a lattice and its co-volume
$\vol (G / \Gamma)$ is the Haar measure $|\Omega|$ of any fundamental domain $\Omega$
under the identification $\mathrm{PSL}(2,\mathbb{R}) / \mathrm{PSO}(2,\mathbb{R}) \simeq \mathbb{C}^+$. With the normalizations used in \cite{abreu2022affine}, the norm $\lVert\cdot\rVert_2$ in Theorem \ref{th_as} is with respect to the Lebesgue measure, and $C_{\psi}$ is the \emph{admissibility constant}
\begin{align*}
C_{\psi} = \int_{0}^\infty |\widehat{\psi}(\xi)|^2\, \frac{d\xi}{\xi}, 
\end{align*}
which indeed equals $\frac{2}{\alpha} {\Vert \psi_{n}^{\alpha }\Vert _{2}^{2}}$ when $\psi$ is the special function $\psi_{n}^{\alpha}$.

Thus, to see that Theorem \ref{th_as} is an application of Theorem \ref{th_rvv} we only need to check that the formal dimension of the representation \eqref{eq:rep_wavelet_space} is precisely
\begin{align}\label{eq_equal}
d_{\tau_n^{\alpha}} = \frac{\Vert \psi_{n}^{\alpha }\Vert _{2}^{2}}{C_{\psi _{n}^{\alpha }}}.
\end{align}
While this is clear from an abstract point of view (see \cite{paul1984affine,bertrand2002characterization, combescure2012coherent}, \cite[Section 9]{romero2022density} and the many references therein), we offer the following computation relying only on elementary wavelet theory (see also \cite{klauder1994wavelets}).

First, as soon as $\psi \in H^2(\mathbb{R})$ is \emph{admissible}, i.e., $C_\psi<\infty$, the orthogonality relations for the wavelet transform (also known as Calder\'on's formula) read
\begin{align} \label{eq_ortho_affine}
	\int_{\mathbb{R}} \int_0^\infty  \langle f_1, \rho(a,b) \psi \rangle \overline{\langle f_2,\rho(a,b) \psi \rangle} \; \frac{da}{a^2} db= C_\psi \langle f_1, f_2 \rangle,
\end{align}
where $f_1, f_2 \in H^2$, the inner products are taken in $H^2$, and $\rho(a,b)f (t) = a^{-1/2} f (a^{-1} (t-b))$. On the other hand, the formal dimension $d_{\tau_n^{\alpha}}$ of the representation \eqref{eq:rep_wavelet_space} is by definition the unique number such that
\begin{align}\label{eq_or}
\begin{aligned}
	&\int_{G}
	\langle \tau^\alpha_n(m) W_{\psi_n^\alpha} f_1,
	W_{\psi_n^\alpha} g_1\rangle
	\overline{
		\langle \tau_n^{\alpha}(m) W_{\psi_n^\alpha} f_2,
		W_{\psi_n^\alpha} g_2\rangle} \,d\mu_G(m)
	\\
	&\qquad=
	\frac{1}{d_{\tau_n^{\alpha}}}
	\langle{W_{\psi_n^\alpha} f_1},{W_{\psi_n^\alpha} f_2}\rangle
	\overline{\langle{W_{\psi_n^\alpha} g_1},{W_{\psi_n^\alpha} g_2}\rangle},
\end{aligned}
\end{align}
for all $f_1, f_2, g_1, g_2 \in H^2$, where $\mu_G$ denotes the Haar measure on $G =  \mathrm{PSL}(2, \mathbb{R})$. 

By definition, $W_{\psi_n^\alpha} f (a,b) = \langle f, \rho(a,b) \psi_n^\alpha \rangle$, and comparison between \eqref{eq_ortho_affine} and \eqref{eq_or} already hints at \eqref{eq_equal}. Let us verify this. 
Note first that under the identification $(a,b) \mapsto ai+b$ the affine group $\mathbb{R}^+ \ltimes \mathbb{R} \simeq \mathbb{C}^+$ embeds into $\mathrm{PSL}(2,\mathbb{R})$ by means of the matrices
\begin{align}\label{eq_2}
	\mathrm{m}_{a,b} =\begin{pmatrix}
		\sqrt{a}&b/\sqrt{a}\\
		0&1/\sqrt{a}
	\end{pmatrix},
	\qquad (a,b) \in \mathbb{R}^+ \ltimes \mathbb{R},
\end{align}
and that these act by $\tau^{\alpha}_n (\mathrm{m}_{a,b}) F (z) = F ( \mathrm{m}^{-1}_{a,b} \cdot z) = F\big(\tfrac{z-b}{a}\big)$. On the other hand, the fundamental property of the wavelet transform is that it intertwines $\rho$ and the regular representation of the affine group on $W_\psi H^2$. Concretely,
\begin{align}\label{eq_inva}
 W_\psi \rho(a,b) f = \tau^{\alpha}_n (\mathrm{m}_{a,b}) W_\psi f.
\end{align}
What is special of the wavelets $\psi^\alpha_n$ is how they interact with rotations \cite{paul1984affine,MR1443157}. In the parlance of Perelomov \cite{perelomov1972coherent}, they are \emph{stationary vectors} of the subgroup $\mathrm{PSO}(2,\mathbb{R})$, cf. \cite[Section 9.1.3]{romero2022density}:
\begin{align}\label{eq_4}
	\tau_n^{\alpha}(\mathrm{r}) W_{\psi^\alpha_n} \psi^\alpha_n = e^{i \phi_{n,\alpha, \mathrm{r}}} W_{\psi^\alpha_n} \psi^\alpha_n, \qquad \mathrm{r} \in 
	\mathrm{PSO}(2, \mathbb{R}),
\end{align}
for some $\phi_{n,\alpha, \mathrm{r}} \in \mathbb{R}$. (Here and above, though we prefer to use the language of representation theory, we only state elementary facts that would not demand the jargon.)

The stationary property \eqref{eq_4} allows us to average \eqref{eq_ortho_affine} over $\mathrm{PSO}(2,\mathbb{R})$ to obtain \eqref{eq_or} with the desired value of $d_{\tau_n^{\alpha}}$. More precisely, taking $f_1=f_2=g_1=g_2=\psi^\alpha_n$ in \eqref{eq_or} and combining  \eqref{eq_ortho_affine}, \eqref{eq_4} and \eqref{eq_inva} we write
\begin{align*}
	C_{\psi_n^{\alpha}}^2 \cdot \| \psi_n^{\alpha} \|_{2}^4&=
	\|{W_{\psi_n^\alpha} \psi^\alpha_n}\|_{L^2}^4
	\\
	&=d_{\tau^\alpha_n}
	\int_{\mathrm{PSL}(2,\mathbb{R})}
	| \langle{\tau^\alpha_n(m) W_{\psi_n^\alpha} \psi_n^\alpha},{W_{\psi_n^\alpha} \psi_n^\alpha}\rangle|^2 \,d\mu_G (m)
	\\
	&=d_{\tau^\alpha_n}
	\int_{\mathbb{R}} \int_0^\infty 
	\int_{\mathrm{PSO}(2,\mathbb{R})}
	| \langle{\tau^\alpha_n(\mathrm{m}_{a,b} \mathrm{r}) W_{\psi_n^\alpha} \psi_n^\alpha},{W_{\psi_n^\alpha} \psi_n^\alpha}\rangle|^2 \, d\mathrm{r} \,\frac{da}{a^2}db
	\\
	&=d_{\tau^\alpha_n}
	\int_{\mathbb{R}} \int_0^\infty 
	\int_{\mathrm{PSO}(2,\mathbb{R})}
	| \langle{\tau^\alpha_n(\mathrm{m}_{a,b}) \tau^\alpha_n(\mathrm{r}) W_{\psi_n^\alpha} \psi_n^\alpha},{W_{\psi_n^\alpha} \psi_n^\alpha}\rangle|^2 \, d\mathrm{r} \,\frac{da}{a^2}db
	\\
	&=d_{\tau^\alpha_n}
	\int_{\mathbb{R}} \int_0^\infty 
	| \langle{\tau^\alpha_n(\mathrm{m}_{a,b}) W_{\psi_n^\alpha} \psi_n^\alpha},{W_{\psi_n^\alpha} \psi_n^\alpha}\rangle|^2 \,\frac{da}{a^2}db
	\\
	&= d_{\tau^\alpha_n}
	\int_{\mathbb{R}} \int_0^\infty 
	| \langle W_{\psi_n^\alpha} \rho(a,b) \psi_n^{\alpha},{W_{\psi_n^\alpha} \psi_n^\alpha}\rangle|^2 \,\frac{da}{a^2}db
	\\
	&= d_{\tau_n^{\alpha}} \cdot C^2_{\psi_n^{\alpha}} \int_{\mathbb{R}} \int_0^\infty  |\langle  \psi_n^{\alpha},  \rho (a,b)  \psi_n^{\alpha} \rangle|^2 \; \frac{da}{a^2}db
	\\
	&= d_{\tau_n^{\alpha}} \cdot C^3_{\psi_n^{\alpha}} \cdot \| \psi_n^{\alpha} \|_{2}^2,
\end{align*}
which gives \eqref{eq_equal}. (Caveat emptor: we tried to follow the normalizations used in \cite{abreu2022affine}, but did not check their consistency.)

\subsection{Further remarks}
Incidentally, the argument above also shows that \eqref{eq:rep_wavelet_space} does preserve the range of the wavelet transform ${W}_{\psi_n^{\alpha}}$. Indeed, as $\tau_n^{\alpha}$ extends the regular representation of the affine group by means of \eqref{eq_2}, it follows that $\mathcal{W}_{\psi_n^{\alpha}}$ is the closed linear span of all elements of the form $\tau^\alpha_n(\mathrm{m}_{a,b})W_{\psi_n^{\alpha}} {\psi_n^{\alpha}}$. 
By \eqref{eq_4}, further acting with a rotation $\mathrm{r}$ results in
$\tau^\alpha_n(\mathrm{r}) \tau^\alpha_n(\mathrm{m}_{a,b})W_{\psi_n^{\alpha}} {\psi_n^{\alpha}} \sim \tau^\alpha_n(\mathrm{r} \mathrm{m}_{a,b}  )W_{\psi_n^{\alpha}} {\psi_n^{\alpha}} \sim \tau^\alpha_n(\mathrm{m}_{c,d} \mathrm{r}')W_{\psi_n^{\alpha}} {\psi_n^{\alpha}} \sim 
\tau^\alpha_n(\mathrm{m}_{c,d})W_{\psi_n^{\alpha}} {\psi_n^{\alpha}}
$, where $\mathrm{r} \mathrm{m}_{a,b}=\mathrm{m}_{c,d} \mathrm{r}'$ with $\mathrm{r}' \in \mathrm{PSO}(2,\mathbb{R})$ and $\sim$ denotes equality up to a unimodular factor.

Finally we point out that the compactness claim in Theorem \ref{th_as} is not correct. Indeed, if $\Gamma \subseteq \mathrm{PSL}(2,\mathbb{R})$ is any Fuchsian group (co-compact or not) satisfying $\vol (G / \Gamma) d_{\tau^\alpha_n} \leq 1$ then there exists $F \in \mathcal{W}_{\psi_n^{\alpha}}$ such that $\{\tau_{n}^{\alpha
}(\gamma)F\}_{\gamma \in \Gamma }$ is a frame for $\mathcal{W}_{\psi
	_{n}^{\alpha }}$; see \cite[Corollaries 3 and 4]{bekka2004square} or \cite[Theorem 8.1]{romero2022density}. See \cite[Example 1]{bekka2004square} for a concrete example of a non co-compact Fuchsian group satisfying the volume bound for appropriate $\alpha$.

\section{A direct proof of Theorem \ref{th_as}}\label{sec_2}
While Theorem \ref{th_as} is a special case of Theorem \ref{th_rvv}, their presentations in \cite{romero2022density} and \cite{abreu2022affine} differ greatly. Our article \cite{romero2022density} is a research survey on what we called ``the density theorem'', where Theorem \ref{th_rvv} is complemented with converse implications, showing that volume conditions also imply the existence of suitable coherent systems. While the density theorem is a basic application of the dimension theory for group von Neumann algebras, \cite{romero2022density} provides an elementary self-contained exposition based on frame theory. For example, Theorem \ref{th_rvv} is proved by a remarkably simple argument due to Janssen in the context of the short-time Fourier transform \cite{janssen_density}, which we now reproduce in the context of Theorem \ref{th_as}.

\begin{proof}[Proof of Theorem \ref{th_as} following \cite{janssen_density, romero2022density}]
(i)
By the orthogonality relations \eqref{eq_or}, for $F,H \in \mathcal{W}_{\psi_n^{\alpha}}$,
\begin{align}\label{eq_x}
d_{\tau_n^{\alpha}}^{-1} \| F \|^2_{L^2} \| H \|^2_{L^2} = \int_G |\langle H, \tau_n^{\alpha} (m) F \rangle |^2 \; d\mu_G (m) = \int_{G/\Gamma} \sum_{\gamma \in \Gamma} |\langle \tau_n^{\alpha}(m)^{-1} H, \tau_n^{\alpha} (\gamma) F \rangle |^2 \; d\mu_{G/\Gamma} (m),
\end{align}
with $d_{\tau_n^{\alpha}}$ given by \eqref{eq_equal}. Therefore, if $(\tau_n^{\alpha} (\gamma) F)_{\gamma \in \Gamma}$ has an upper frame bound $B > 0$, then $d_{\tau^{\alpha}_n}^{-1} \| F \|_{L^2}^2 \leq B \vol(G/\Gamma)$. Similarly, if $(\tau_n^{\alpha} (\gamma) F)_{\gamma \in \Gamma}$ has a lower frame bound $A > 0$, then $\vol(G/\Gamma) A \leq d_{\tau^{\alpha}_n}^{-1} \| F \|_{L^2}^2 $.

(ii) Suppose $(\tau_n^{\alpha} (\gamma) F)_{\gamma \in \Gamma}$ is a frame. Then the frame operator $S := \sum_{\gamma \in \Gamma} \langle \cdot , \tau_n^{\alpha} (\gamma) F \rangle \tau_n^{\alpha} (\gamma) F$ is bounded and invertible on $\mathcal{W}_{\psi_n^{\alpha}}$, and commutes with  $\tau_n^{\alpha} (\gamma)$ for all $\gamma \in \Gamma$ (even if $\tau_n^{\alpha}$ is only a projective representation!). Therefore, $
(\tau_n^{\alpha} (\gamma) S^{-1/2} F )_{\gamma \in \Gamma} = ( S^{-1/2} \tau_n^{\alpha} (\gamma)  F )_{\gamma \in \Gamma}
$ is a frame with frame bounds $A=B=1$. Hence, $\| S^{-1/2} F \| \leq 1$, and (i) implies that $\vol(G/\Gamma) \leq d_{\tau_n^{\alpha}}^{-1}$.

(iii) Suppose $(\tau_n^{\alpha} (\gamma) F)_{\gamma \in \Gamma}$ is a Riesz sequence. Then the frame operator $S$ is bounded and invertible on $\mathcal{K} := \overline{\Span \{ \tau_n^{\alpha} (\gamma) F : \gamma \in \Gamma \}}$ and $
(\tau_n^{\alpha} (\gamma) S^{-1/2} F )_{\gamma \in \Gamma} = ( S^{-1/2} \tau_n^{\alpha} (\gamma)  F )_{\gamma \in \Gamma}
$ is an orthonormal sequence in $\mathcal{W}_{\psi_n^{\alpha}}$. In particular, it has an upper frame bound $B = 1$ and $\| S^{-1/2} F \|_{L^2} = 1$, and (i) shows that $\vol(G/\Gamma) \geq d_{\tau_n^{\alpha}}^{-1}$.
\end{proof}
Note that the above proof invokes \eqref{eq_equal}.  Alternatively, one can invoke \eqref{eq_ortho_affine} and average over $PSO(2,\mathbb{R})$ to directly obtain \eqref{eq_x} with the prescribed value of $d_{\tau_n^{\alpha}}$.

\section{The literature and attribution of results}
We now offer contrasting views to the presentation of the literature in \cite{abreu2022affine} and take exception to certain attributions.

\subsection{The status of the problem}
Theorem \ref{th_as} is presented in \cite{abreu2022affine} as a solution to a longstanding problem of Perelomov. While Perelomov did mention the problem of relating the spanning properties of coherent systems to the co-volume of the generating lattice, the answer provided by the coupling theory of group von Neumann algebras has been known for a long time. For example, Bekka \cite[Remark 2.10]{bekka2000square} describes the statement in part (i) of Theorem \ref{th_rvv} as an answer to Perelomov's question. The statement in part (ii) of Theorem \ref{th_rvv} seems to be, on the other hand, a moderately original contribution of \cite{romero2022density}, whereas the related notion of separating vector is more common in the operator algebra literature (see \cite[Proposition 5.2]{romero2022density}.) Further historical remarks can be found in \cite[Section 1 and Section 9.1.1, ``Perelomov’s uniqueness problem'']{romero2022density}.

Second, Sections 1.1 and 1.2 of \cite{abreu2022affine} elaborate on the many technical challenges that the exponential growth of balls in the hyperbolic metric presents when attempting to formulate a density theory for coherent systems. Exponential growth is indeed an obstacle to many common approaches to frame density results, as these require being able to neglect boundary interactions at large scales. However, as evidenced by \eqref{eq_x}, such technical challenges simply do not arise in the context of Theorems \ref{th_rvv} and \ref{th_as}, since the boundaries of the different tiles in the partition of the upper-half plane induced by a lattice do not interact.

\subsection{Our work}
Firstly, our research survey \cite{romero2022density} discusses at length the application of Theorem \ref{th_rvv} - and its counterpart concerning the existence of coherent systems - to the holomorphic series of $\mathrm{PSL}(2,\mathbb{R})$, which exhaust up to complex conjugation and projective unitary equivalence all square integrable irreducible representations of $\mathrm{PSL}(2,\mathbb{R})$. The statements in Theorem \ref{th_as} are of course invariant under unitary equivalence.

Secondly, the special case of Theorem \ref{th_as} with $F=W_{\psi_n^{\alpha}}\psi_n^{\alpha}$ is attributed to us in \cite{abreu2022affine}, whereas the reader is led to understand that the generalization to arbitrary $F$ is the original contribution of \cite{abreu2022affine}.
While, indeed, in \cite[Section 9.1.3, ``Perelomov's problem with respect to other special vectors'']{romero2022density} we discuss the application of Theorem \ref{th_rvv} to the special vectors in question, the very same arguments apply to any other generator. The reason why we chose to discuss rotationally symmetric generators is that, in our opinion, this is the case where the resulting coherent system can be called an \emph{affine system}, since, as explained in
\cite{klauder1994wavelets}, rotations can be factored out, and the orbit can be described up to constant factors in terms of the affine maps \eqref{eq_2}.

Thirdly, with a certain physical motivation, the authors of \cite{ABDM} consider the above-mentioned affine coherent states and conclude that the density condition 
\begin{align*}
|\Omega| \leq \frac{4 (n+1)}{\alpha} 
\end{align*}
is necessary for completeness. In \cite[Section 9.1.3]{romero2022density} we point out that coupling theory, as embodied in Theorem \ref{th_rvv}, offers the sharper bound
\begin{align*}
	|\Omega| \leq \frac{2}{\alpha}. 
\end{align*}
The final remark in \cite[Section 2]{abreu2022affine} on the ``physical relevance of the results'' may give the impression that this observation is original to \cite{abreu2022affine}; we invite the interested reader to look into \cite[Section 9.1.3]{romero2022density}.

\section*{Acknowledgment}
J. L. R. gratefully acknowledges support from the Austrian Science Fund (FWF): Y 1199.

\end{document}